\documentclass[a4paper,oneside,10pt]{article}%
\usepackage{amsmath}
\usepackage{amsfonts}
\usepackage{amssymb}
\usepackage{graphicx}
\usepackage[square,numbers,sort&compress]{natbib}%
\providecommand{\U}[1]{\protect\rule{.1in}{.1in}}

\newtheorem{theorem}{Theorem}[section]

\newtheorem{corollary}[theorem]{Corollary}

\newtheorem{definition}[theorem]{Definition}

\newtheorem{example}[theorem]{Example}

\newtheorem{lemma}[theorem]{Lemma}

\newtheorem{problem}[theorem]{Problem}
\newtheorem{proposition}[theorem]{Proposition}
\newtheorem{remark}[theorem]{Remark}

\newenvironment{proof}[1][Proof]{\noindent\textbf{#1.} }{\ \rule{0.5em}{0.5em}}
\numberwithin{equation}{section}

\begin{document}

\title{The minimum mean square estimator for a sublinear operator}
\author{Ji Shaolin\thanks{Institute for Financial Studies and Institute of
Mathematics, Shandong University, Jinan, Shandong 250100, PR China
(Jsl@sdu.edu.cn, Fax: +86 0531 88564100).}
\and Sun Chuanfeng\thanks{Institute for Financial Studies and Institute of
Mathematics, Shandong University, Jinan, Shandong 250100, PR China.}}
\maketitle

\begin{abstract}
In this paper, we study the minimum mean square estimator for a sublinear
operator. Under some mild assumptions, we prove the existence and uniqueness
of the minimum mean square estimator.  Several characterizations of the
minimum mean square estimator are obtained. We also explore the relationship
between the minimum mean square estimator and the conditional coherent risk
measure and conditional $g$-expectation.

\end{abstract}


\section{Introduction}

In the classical probability theory, the conditional expectation of a random
variable $\xi$ in $L_{\mathcal{F}}^{2}(\Omega)$ is just the minimum mean
square estimator. In more details, let $\mathcal{C}$ be a sub $\sigma$-algebra
of $\mathcal{F}$. Then this minimum mean square estimator is  the projection
of $\xi$ from $L_{\mathcal{F}}^{2}(\Omega)$ to $L_{\mathcal{C}}^{2}(\Omega)$.
Therefore, the minimum mean square estimator can be used as an alternative
definition of conditional expectation.

In recent decades, nonlinear risk measures and nonlinear expectations have
been proposed and developed rapidly. Various definitions of conditional
nonlinear expectations (risk measures) are proposed. For example, Artzner et
al. [2] introduced coherent risk measure theory in which the conditional
expectation (conditional risk measure) is defined as
\[
\bar{\Phi}_{t}[\xi]:=\mathop{\textrm{ess}\sup}_{P\in\mathcal{P}}E_{P}%
[\xi|\mathcal{F}_{t}],
\]
where $\xi\in\mathcal{F}_{T}$, $\mathcal{F}_{t}$ is a sub $\sigma$-algebra of
$\mathcal{F}_{T}$ and $\mathcal{P}$ is a family of probability measures. They
have proved that if $\mathcal{P}$ is `stable', then the conditional
expectation defined above is time consistent. It is also well known that Peng
studied $g$-expectation in \cite{Peng0} and defined the conditional $g$-expectation
$\mathcal{E}_{g}[\xi|\mathcal{F}_{t}]$ as the solution of a backward
stochastic differential equation with the generator $g$ and the terminal value
$\xi$ at time $t$. $g$-expectation has many good properties including time
consistency, i.e., $\forall0\leq s\leq t\leq T$, we have $\mathcal{E}%
_{g}[\mathcal{E}_{g}[\xi|\mathcal{F}_{t}]|\mathcal{F}_{s}]=\mathcal{E}_{g}%
[\xi|\mathcal{F}_{s}]$.

So it is interesting to explore whether conditional nonlinear expectations
still coincide with the minimum mean square estimators. Note that many
interesting nonlinear expectations (risk measures) are sublinear. So the
purpose of this paper is to study the minimum mean square estimator for a
sublinear operator. We will show that for this minimum mean square estimator,
generally speaking, the time consistency property does not hold. In accordance
with this result, both of the above conditional nonlinear expectations
(Artzner et al.'s and Peng's) fail to be the minimum mean square estimators.

The paper is organized as follows. In section 2, we formulate our problem.
Under some mild assumptions, we prove the existence and uniqueness of the
minimum mean square estimator in section 3. In section 4, we obtain several
characterizations of the minimum mean square estimator. At last section, we
first give the basic properties of the minimum mean square estimator. Then we
explore the relationship between the minimum mean square estimator and the
conditional coherent risk measure and conditional $g$-expectation.

\section{{Problem formulation}}

\bigskip

\subsection{{Preliminary}}

For a given measurable space $(\Omega,\mathcal{F})$, we denote all the bounded
$\mathcal{F}$-measurable functions by $\mathbb{F}$.

\begin{definition}
\label{def-sublinear-operator} A sublinear operator $\rho$ is a functional
$\rho:\mathbb{F}\mapsto\mathbb{R}$ satisfying

(i) Monotonicity: $\rho(\xi_{1})\geq\rho(\xi_{2})$ if $\xi_{1}\geq\xi_{2}$;

(ii) Constant preserving: $\rho(c)=c$ for $c\in\mathbb{R}$;

(iii) Sub-additivity: For each $\xi_{1},\xi_{2}\in\mathbb{F}$, $\rho(\xi
_{1}+\xi_{2})\leq\rho(\xi_{1})+\rho(\xi_{2})$;

(iv) Positive homogeneity: $\rho(\lambda\xi)=\lambda\rho(\xi)$ for
$\lambda\geq0$.
\end{definition}

Note that $\mathbb{F}$ is a Banach space endowed with the supremum norm.
Denote the dual space of $\mathbb{F}$ by $\mathbb{F}^{\ast}$. It is well known
that there is a one-to-one correspondence between $\mathbb{F}^{\ast}$ and the
class of additive set functions. Then we denote an element in $\mathbb{F}%
^{\ast}$ by $E_{P}$ where $P$ is an additive set function. Sometimes we also
use $P$ instead of $E_{P}$.

\begin{proposition}
\label{prop-fatou lemma} Suppose that the sublinear operator $\rho$ can be
represented by a family of probability measures $\mathcal{P}$, i.e., $\rho
(\xi)=\underset{P\in\mathcal{P}}{\sup}E_{P}[\xi]$. For a sequence $\{\xi
_{n}\}_{n\in\mathbb{N}}$, if there exists a $M\in\mathbb{R}$ such that
$\xi_{n}\geq M$ for all $n$, then we have
\[
\rho(\mathop{\lim\inf}_{n}\xi_{n})\leq\mathop{\lim\inf}_{n}\rho(\xi_{n}).
\]

\end{proposition}

\begin{proof}
Set $\zeta_{n}=\inf_{k\geq n}\xi_{k}$. Then $\zeta_{n}\leq\xi_{n}$ and
$\{\zeta_{n}\}_{n\in\mathbb{N}}$ is an increasing sequence. It is easy to see
that
\[
\rho(\mathop{\lim\inf}_{n}\xi_{n})=\rho(\lim_{n}\zeta_{n})=\lim_{n}\rho
(\zeta_{n})\leq\mathop{\lim\inf}_{n}\rho(\xi_{n}).
\]
This completes the proof.
\end{proof}

\begin{definition}
\label{def-fatou property} We call a sublinear operator $\rho$ continuous from
above on $\mathcal{F}$ if for each sequence $\{\xi_{n}\}_{n\in\mathbb{N}%
}\subset\mathbb{F}$ satisfying $\xi_{n}\downarrow0$, we have
\[
\rho(\xi_{n})\downarrow0.
\]

\end{definition}

\begin{lemma}
\label{lem-fatou-probility-property} If a sublinear operator $\rho$ is
continuous from above on $\mathcal{F}$, then for any linear operator $E_{P}$
dominated by $\rho$, $P$ is a probability measure.
\end{lemma}

\begin{proof}
For any $A_{n}\downarrow\phi$, we have $\rho(I_{A_{n}})\downarrow0$. If a
linear operator $E_{P}$ is dominated by $\rho$, then $P(A_{n})\downarrow0$. It
is easy to see that $P(\Omega)=1$. Thus, $P$ is a probability measure.
\end{proof}

\begin{proposition}
\label{prop-fatou property} A sublinear operator $\rho$ is continuous from
above on $\mathcal{F}$ if and only if there exists a probability measure
$P_{0}$ and a family of probability measures $\mathcal{P}$ such that

i) $\rho(X)=\sup_{P\in\mathcal{P}}E_{P}[X]$ for all $X\in\mathbb{F}$;

ii) any element in $\mathcal{P}$ is absolutely continuous with respect to
$P_{0}$;

iii) the set $\{\frac{dP}{dP_{0}};P\in\mathcal{P}\}$ is $\sigma(L^{1}%
(P_{0}),L^{\infty}(P_{0}))$-compact,\newline where $\sigma(L^{1}%
(P_{0}),L^{\infty}(P_{0}))$ denotes the weak topology defined on $L^{1}%
(P_{0})$.
\end{proposition}

\begin{proof}
$\Rightarrow$ By Theorem \ref{theA.1} in Appendix A, $\rho$ can be represented
by the family of linear operators dominated by $\rho$. We denote by
$\mathcal{P}$ all the linear operators dominated by $\rho$. Since $\rho$ is
continuous from above on $\mathcal{F}$, by Lemma
\ref{lem-fatou-probility-property}, every element in $\mathcal{P}$ is a
probability measure. By Theorem \ref{theA.2}, $\mathcal{P}$ is $\sigma
(\mathbb{F}^{\ast},\mathbb{F})$-compact, where $\sigma(\mathbb{F}^{\ast
},\mathbb{F})$ denotes the $weak^{\ast}$ topology defined on $\mathbb{F}%
^{\ast}$. By Theorem \ref{theA.3}, there exists a $P_{0}\in\mathbb{F}%
_{c}^{\ast}$ such that all the elements in $\mathcal{P}$ are absolutely
continuous with respect to $P_{0}$, where $\mathbb{F}_{c}^{\ast}$ denotes all
the countably additive measures in $\mathbb{F}^{\ast}$. Since $\rho$ is
continuous from above on $\mathcal{F}$ and the dual space of $L^{1}(P_{0})$ is
$L^{\infty}(P_{0})$, by Corollary 4.35 in \cite{FS}, the set $\{\frac
{dP}{dP_{0}};P\in\mathcal{P}\}$ is $\sigma(L^{1}(P_{0}),L^{\infty}(P_{0}%
))$-compact, where $L^{1}(P_{0})$ is the space of integrable random variables
and $L^{\infty}(P_{0})$ is the space of all equivalence classes of bounded
real valued random variables.

$\Leftarrow$ We directly deduce this result by Dini's theorem.
\end{proof}

In the following, we will denote by $\mathcal{P}$\ all the linear operators
dominated by $\rho$.

\begin{definition}
\label{def-proper} We call a sublinear operator $\rho$ proper if all the
elements in $\mathcal{P}$ are equivalent to $P_{0}$, where $P_{0}$ is the
probability measure in Proposition \ref{prop-fatou property}.
\end{definition}

\subsection{Minimum mean square estimator}

Let $\mathcal{C}$ be a sub $\sigma$-algebra of $\mathcal{F}$ and $\mathbb{C}$
be the set of all the bounded $\mathcal{C}$-measurable functions. For a given
$\xi\in\mathbb{F}$, our problem is to solve its minimum mean square estimator
for the sublinear operator $\rho$ when we only know "the information"
$\mathcal{C}$. In more details, we want to solve the following optimization problem.

\begin{problem}
\label{Problem} Find a $\hat{\eta}\in\mathbb{C}$ such that
\begin{equation}
\rho(\xi-\hat{\eta})^{2}=\inf_{\eta\in\mathbb{C}}\rho(\xi-\eta)^{2}.
\label{problem}%
\end{equation}

\end{problem}

The optimal solution $\hat{\eta}$ of (\ref{problem}) is called the minimum
mean square estimator. It is also regarded as a minimax estimator in
statistical decision theory.

If $\rho$ degenerates to a linear operator, then $\mathcal{P}$\ contains only
one probability measure and $\rho$ is the mathematical expectation under this
probability measure. In this case, it is well known that the minimum mean
square estimator $\hat{\eta}$ is just the conditional expectation $E[\xi
\mid\mathcal{C]}$.

\section{{Existence and uniqueness results}\newline}

In this section, we study the existence and uniqueness of the minimum mean
square estimator.

For a given $\xi\in\mathbb{F}$, we always suppose that $\sup|\xi|\leq M$ where
$M$ is a positive constant.

\subsection{Existence}

\begin{lemma}
\label{lem-reformulate prob} Suppose that $\xi\in\mathbb{F}$. Then we have
\[
\inf_{\eta\in\mathbb{C}}\rho(\xi-\eta)^{2}=\inf_{\eta\in\mathbb{G}}\rho
(\xi-\eta)^{2},
\]
where $\mathbb{G}$ is all the $\mathcal{C}$-measurable functions bounded by
$M$.
\end{lemma}

\begin{proof}
For any $\eta\in\mathbb{C}$, let $\bar{\eta}:=\eta\mathbf{I}_{\{-M\leq\eta\leq
M\}}+M\mathbf{I}_{\{\eta>M\}}-M\mathbf{I}_{\{\eta<-M\}}$. Then $\bar{\eta}%
\in\mathbb{G}$. For any $P\in\mathcal{P}$, we have
\[%
\begin{array}
[c]{r@{}l}
& E_{P}[(\xi-\eta)^{2}]-E_{P}[(\xi-\bar{\eta})^{2}]=E_{P}[(\bar{\eta}%
-\eta)(2\xi-\eta-\bar{\eta})]\\
\geq & E_{P}[(M-\eta)(2\xi-2M)\mathbf{I}_{\{\eta>M\}}]+E_{P}[(-M-\eta
)(2\xi+2M)\mathbf{I}_{\{\eta<-M\}}].
\end{array}
\]
Since $-M\leq\xi\leq M$, we have
\[
E_{P}[(\xi-\eta)^{2}]\geq E_{P}[(\xi-\bar{\eta})^{2}].
\]
It yields that
\[
\rho(\xi-\eta)^{2}\geq\rho(\xi-\bar{\eta})^{2}%
\]
and
\[
\inf_{\eta\in\mathbb{C}}\rho(\xi-\eta)^{2}\geq\inf_{\eta\in\mathbb{G}}\rho
(\xi-\eta)^{2}.
\]

On the other hand, since $\mathbb{G}\subset\mathbb{C}$, the inverse inequality
is obviously true.
\end{proof}

\begin{theorem}
\label{exist} If the sublinear operator $\rho$ is continuous from above on
$\mathcal{F}$, then there exists an optimal solution $\hat{\eta}\in\mathbb{G}$
for problem \ref{Problem}.
\end{theorem}

\begin{proof}
By Lemma \ref{lem-reformulate prob}, there exists a sequence $\{\eta_{n}%
;n\in\mathbb{N}\}\subset\mathbb{G}$ such that
\[
\rho(\xi-\eta_{n})^{2}<\alpha+\frac{1}{2^{n}},
\]
where $\alpha:=\inf_{\eta\in\mathbb{C}}\rho(\xi-\eta)^{2}$. By Koml\'{o}s theorem, there exists a subsequence $\{\eta_{n_{i}}\}_{i\geq1}$ and a
random variable $\hat{\eta}$ such that
\[
\lim_{k\rightarrow\infty}\frac{1}{k}\sum_{i=1}^{k}\eta_{n_{i}}=\hat{\eta
}.\quad P_{0}-a.s.
\]
Since $\{\eta_{n}\}_{n\geq1}$ is bounded by $M$, then $\hat{\eta}\in
\mathbb{G}$. By Proposition \ref{prop-fatou lemma}, we have
\[
\rho(\xi-\hat{\eta})^{2}\leq\liminf_{k\rightarrow\infty}\frac{1}{k}\sum
_{i=1}^{k}\rho(\xi-\eta_{n_{i}})^{2}\leq\lim_{k\rightarrow\infty}(\alpha
+\frac{1}{k})=\alpha.
\]
This completes the proof.
\end{proof}

\begin{remark}
\label{rem3.1} If we only assume that there exists a probability measure
$P_{0}$ such that the $\rho$ is a sublinear operator generated by a family of
probability measures which are all absolutely continuous with respect to
$P_{0}$, then Theorem \ref{exist} still holds.
\end{remark}

\subsection{Uniqueness}

In this subsection, we prove that there exists a unique optimal solution of
problem \ref{Problem}.

\begin{lemma}
\label{lem-attainable-max} If the sublinear operator $\rho$ is continuous from
above on $\mathcal{F}$, then for a given $\xi\in\mathbb{F}$, we have
\[
\sup_{P\in\mathcal{P}}\inf_{\eta\in\mathbb{G}}E_{P}[(\xi-\eta)^{2}]=\max
_{P\in\mathcal{P}}\inf_{\eta\in\mathbb{G}}E_{P}[(\xi-\eta)^{2}].
\]

\end{lemma}

\begin{proof}
By Proposition \ref{prop-fatou property}, there exists a probability measure
$P_{0}$ such that $P\ll P_{0}$ for all $P\in\mathcal{P}$. Let $f_{P}%
:=\frac{dP}{dP_{0}}$ and
\[
\beta:=\sup_{P\in\mathcal{P}}\inf_{\eta\in\mathbb{G}}E_{P}[(\xi-\eta
)^{2}]=\sup_{P\in\mathcal{P}}\inf_{\eta\in\mathbb{G}}E_{P_{0}}[f_{P}(\xi
-\eta)^{2}].
\]
Then take a sequence $\{f_{P_{n}};P_{n}\in\mathcal{P}\}_{n\geq1}$ such that
\[
\inf_{\eta\in\mathbb{G}}E_{P_{0}}[f_{P_{n}}(\xi-\eta)^{2}]\geq\beta-\frac
{1}{2^{n}}.
\]
By Koml\'{o}s theorem, there exists a subsequence
$\{f_{P_{n_{i}}}\}_{i\geq1}$ of $\{f_{P_{n}}\}_{n\geq1}$ and a random variable
$f_{\hat{P}}\in L^{1}(\Omega,\mathcal{F},P_{0})$ such that
\[
\lim_{k\rightarrow\infty}\frac{1}{k}\sum_{i=1}^{k}f_{P_{n_{i}}}=f_{\hat{P}%
}\quad P_{0}-a.s..
\]
Set $g_{k}:=\frac{1}{k}\sum_{i=1}^{k}f_{P_{n_{i}}}$. Then $g_{k}\in
\{f_{P};P\in\mathcal{P}\}$. By Proposition \ref{prop-fatou property}, we know
$\{f_{P};P\in\mathcal{P}\}$ is $\sigma(L^{1}(P_{0}),L^{\infty}(P_{0}%
))$-compact. By Dunford-Pettis theorem, it is uniformly integrable. Thus
$\{g_{k}\}_{k\geq1}$ is also uniformly integrable and $||g_{k}-f_{\hat{P}%
}||_{L^{1}(P_{0})}\rightarrow0$. This shows $\hat{P}\in\mathcal{P}$.

On the other hand, for any $\eta\in\mathbb{G}$ and $k\in\mathbb{N}$, we have
\[
E_{P_{0}}[g_{k}(\xi-\eta)^{2}]\geq\inf_{\tilde{\eta}\in\mathbb{G}}E_{P_{0}%
}[g_{k}(\xi-\tilde{\eta})^{2}].
\]
Then for any $\eta\in\mathbb{G}$, we have
\[
\lim_{k\rightarrow\infty}E_{P_{0}}[g_{k}(\xi-\eta)^{2}]\geq
\mathop{\lim\sup}_{k\rightarrow\infty}\inf_{\tilde{\eta}\in\mathbb{G}}%
E_{P_{0}}[g_{k}(\xi-\tilde{\eta})^{2}].
\]
Thus
\[
\inf_{\eta\in\mathbb{G}}\lim_{k\rightarrow\infty}E_{P_{0}}[g_{k}(\xi-\eta
)^{2}]\geq\mathop{\lim\sup}_{k\rightarrow\infty}\inf_{\tilde{\eta}%
\in\mathbb{G}}E_{P_{0}}[g_{k}(\xi-\tilde{\eta})^{2}].
\]
Since $\{g_{k}\}_{k\geq1}$ is uniformly integral and $||(\xi-\eta
)^{2}||_{L^{\infty}}\leq4M^{2}$, we have
\[
\inf_{\eta\in\mathbb{G}}E_{P_{0}}[f_{\hat{P}}(\xi-\eta)^{2}]=\inf_{\eta
\in\mathbb{G}}E_{P_{0}}[\lim_{k\rightarrow\infty}g_{k}(\xi-\eta)^{2}%
]=\inf_{\eta\in\mathbb{G}}\lim_{k\rightarrow\infty}E_{P_{0}}[g_{k}(\xi
-\eta)^{2}].
\]
Then
\[
\inf_{\eta\in\mathbb{G}}E_{P_{0}}[f_{\hat{P}}(\xi-\eta)^{2}]\geq
\mathop{\lim\sup}_{k\rightarrow\infty}\frac{1}{k}\sum_{i=1}^{k}\inf_{\eta
\in\mathbb{G}}E_{P_{0}}[f_{n_{i}}(\xi-\eta)^{2}]\geq\beta.
\]
Since $\hat{P}\in\mathcal{P}$, we have
\[
\inf_{\eta\in\mathbb{G}}E_{\hat{P}}[(\xi-\eta)^{2}]=\sup_{P\in\mathcal{P}}%
\inf_{\eta\in\mathbb{G}}E_{P}[(\xi-\eta)^{2}].
\]
This completes the proof.
\end{proof}

\begin{corollary}
\label{cor-attainable-max-1} If the sublinear operator $\rho$ is continuous
from above on $\mathcal{F}$, then for a given $\xi\in\mathbb{F}$, we have
\[
\sup_{P\in\mathcal{P}}\inf_{\eta\in\mathbb{C}}E_{P}[(\xi-\eta)^{2}]=\max
_{P\in\mathcal{P}}\inf_{\eta\in\mathbb{C}}E_{P}[(\xi-\eta)^{2}].
\]

\end{corollary}

\begin{proof}
Choose $\hat{P}$ as in Lemma \ref{lem-attainable-max}. By Lemma
\ref{lem-reformulate prob} and Lemma \ref{lem-attainable-max}, we have
\[
\sup_{P\in\mathcal{P}}\inf_{\eta\in\mathbb{C}}E_{P}[(\xi-\eta)^{2}]\leq
\sup_{P\in\mathcal{P}}\inf_{\eta\in\mathbb{G}}E_{P}[(\xi-\eta)^{2}]=\inf
_{\eta\in\mathbb{G}}E_{\hat{P}}[(\xi-\eta)^{2}]=\inf_{\eta\in\mathbb{C}%
}E_{\hat{P}}[(\xi-\eta)^{2}].
\]
On the other hand, the inverse inequality is obvious. Then
\[
\inf_{\eta\in\mathbb{C}}E_{\hat{P}}[(\xi-\eta)^{2}]=\sup_{P\in\mathcal{P}}%
\inf_{\eta\in\mathbb{C}}E_{P}[(\xi-\eta)^{2}].
\]
Since $\hat{P}\in\mathcal{P}$, we have
\[
\sup_{P\in\mathcal{P}}\inf_{\eta\in\mathbb{C}}E_{P}[(\xi-\eta)^{2}]=\max
_{P\in\mathcal{P}}\inf_{\eta\in\mathbb{C}}E_{P}[(\xi-\eta)^{2}].
\]
This completes the proof.
\end{proof}

\begin{theorem}
\label{unique} If the sublinear operator $\rho$ is continuous from above on
$\mathcal{F}$ and proper, then for any given $\xi\in\mathbb{F}$, there exists
a unique optimal solution of problem \ref{Problem}.
\end{theorem}

\begin{proof}
The existence result is proved in Theorem \ref{exist}. Now we prove the
optimal solution is unique. Since $\mathcal{P}$ is $\sigma(L^{1}%
(P_{0}),L^{\infty}(P_{0}))$-compact, by Minimax theorem, we have
\[
\inf_{\eta\in\mathbb{C}}\sup_{P\in\mathcal{P}}E_{P}[(\xi-\eta)^{2}]=\sup
_{P\in\mathcal{P}}\inf_{\eta\in\mathbb{C}}E_{P}[(\xi-\eta)^{2}].
\]
Since the optimal solution exists, by Corollary \ref{cor-attainable-max-1}, we
have
\[
\min_{\eta\in\mathbb{C}}\sup_{P\in\mathcal{P}}E_{P}[(\xi-\eta)^{2}]=\max
_{P\in\mathcal{P}}\inf_{\eta\in\mathbb{C}}E_{P}[(\xi-\eta)^{2}].
\]
Let $\hat{\eta}$ be an optimal solution and $\hat{P}$ as in Corollary
\ref{cor-attainable-max-1}. By Minimax theorem, $(\hat{\eta},\hat{P})$
is an saddle point, i.e.,
\[
E_{P}[(\xi-\hat{\eta})^{2}]\leq E_{\hat{P}}[(\xi-\hat{\eta})^{2}]\leq
E_{\hat{P}}[(\xi-\eta)^{2}],\quad\forall P\in\mathcal{P},\eta\in\mathbb{C}.
\]
This result shows that if $\hat{\eta}$ is an optimal solution, then there
exists a $\hat{P}\in\mathcal{P}$ such that $\hat{\eta}=E_{\hat{P}}%
[\xi|\mathcal{C}]$.

Suppose that there exist two optimal solutions $\hat{\eta}_{1}$ and $\hat
{\eta}_{2}$. Denote the accompanying probabilities by $\hat{P}_{1}$ and
$\hat{P}_{2}$ respectively. Then $\hat{\eta}_{1}=E_{\hat{P}_{1}}%
[\xi|\mathcal{C}]$ and $\hat{\eta}_{2}=E_{\hat{P}_{2}}[\xi|\mathcal{C}]$. Set
$P^{\lambda}:=\lambda\hat{P}_{1}+(1-\lambda)\hat{P}_{2}$, $\lambda\in(0,1)$.
Denote $\lambda E_{P^{\lambda}}[\frac{d\hat{P}_{1}}{dP^{\lambda}}%
|\mathcal{C}]$ by $\lambda_{\hat{P}_{1}}$ and $(1-\lambda)E_{P^{\lambda}%
}[\frac{d\hat{P}_{2}}{dP^{\lambda}}|\mathcal{C}]$ by $\lambda_{\hat{P}_{2}}$.
It is easy to see that $\lambda_{\hat{P}_{1}}+\lambda_{\hat{P}_{2}}=1$.
\[%
\begin{array}
[c]{l@{}l}
& E_{P^{\lambda}}[(\xi-\lambda_{\hat{P}_{1}}\hat{\eta}_{1}-\lambda_{\hat
{P}_{2}}\hat{\eta}_{2})^{2}]\\
= & E_{P^{\lambda}}[(\lambda_{\hat{P}_{1}}(\xi-\hat{\eta}_{1})+\lambda
_{\hat{P}_{2}}(\xi-\hat{\eta}_{2}))^{2}]\\
= & E_{P^{\lambda}}[\lambda_{\hat{P}_{1}}^{2}(\xi-\hat{\eta}_{1})^{2}%
+\lambda_{\hat{P}_{2}}^{2}(\xi-\hat{\eta}_{2})^{2}+2\lambda_{\hat{P}_{1}%
}\lambda_{\hat{P}_{2}}(\xi-\hat{\eta}_{1})(\xi-\hat{\eta}_{2})]\\
= & E_{P^{\lambda}}[\lambda_{\hat{P}_{1}}(\xi-\hat{\eta}_{1})^{2}%
+\lambda_{\hat{P}_{2}}(\xi-\hat{\eta}_{2})^{2}-\lambda_{\hat{P}_{1}}%
\lambda_{\hat{P}_{2}}(\hat{\eta}_{1}-\hat{\eta}_{2})^{2}]\\
= & \lambda E_{\hat{P}_{1}}[(\xi-\hat{\eta}_{1})^{2}]+(1-\lambda)E_{\hat
{P}_{2}}[(\xi-\hat{\eta}_{2})^{2}]\\
& +\lambda E_{\hat{P}_{1}}[\lambda_{\hat{P}_{2}}^{2}(\hat{\eta}_{1}-\hat{\eta
}_{2})^{2}]+(1-\lambda)E_{\hat{P}_{2}}[\lambda_{\hat{P}_{1}}^{2}(\hat{\eta
}_{1}-\hat{\eta}_{2})^{2}]\\
\geq & \alpha,
\end{array}
\]
where $\alpha:=\inf_{\eta\in\mathbb{C}}\rho(\xi-\eta)^{2}$.

Since $\rho$ is proper, we have that $E_{P^{\lambda}}[(\xi-\lambda_{\hat
{P}_{1}}\hat{\eta}_{1}-\lambda_{\hat{P}_{2}}\hat{\eta}_{2})^{2}]=\alpha$ if
and only if $\hat{\eta}_{1}=\hat{\eta}_{2}$ $P_{0}$-a.s..

On the other hand, since
\[
\hat{\eta}_{1}=E_{\hat{P}_{1}}[\xi|\mathcal{C}]=\frac{E_{P^{\lambda}}[\xi
\frac{d\hat{P}_{1}}{dP^{\lambda}}|\mathcal{C}]}{E_{P^{\lambda}}[\frac{d\hat
{P}_{1}}{dP^{\lambda}}|\mathcal{C}]}%
\]
and
\[
\hat{\eta}_{2}=E_{\hat{P}_{2}}[\xi|\mathcal{C}]=\frac{E_{P^{\lambda}}[\xi
\frac{d\hat{P}_{2}}{dP^{\lambda}}|\mathcal{C}]}{E_{P^{\lambda}}[\frac{d\hat
{P}_{2}}{dP^{\lambda}}|\mathcal{C}]},
\]
we deduce that $\lambda_{P_{1}}\hat{\eta}_{1}+\lambda_{P_{2}}\hat{\eta}%
_{2}=E_{P^{\lambda}}[\xi|\mathcal{C}]$. Since $(\hat{\eta}_{1},\hat{P}_{1})$
is a saddle point, we have
\[
E_{P^{\lambda}}[(\xi-E_{P^{\lambda}}[\xi|\mathcal{C}])^{2}]\leq E_{P^{\lambda
}}[(\xi-\hat{\eta}_{1})^{2}]\leq E_{\hat{P}_{1}}[(\xi-\hat{\eta}_{1}%
)^{2}]=\alpha.
\]
It yields that $E_{P^{\lambda}}[(\xi-E_{P^{\lambda}}[\xi|\mathcal{C}%
])^{2}]=\alpha$. Thus, $\hat{\eta}_{1}=\hat{\eta}_{2}$ $P_{0}$-a.s..
\end{proof}

\section{Characterizations of the minimum mean square estimator}

In this section, we obtain several characterizations of the minimum mean
square estimator.

\subsection{The orthogonal projection}

If $\mathcal{P}$\ contains only one probability $P$, then by probability
theory, the minimum mean square estimator $\hat{\eta}$ is just the conditional
expectation $E_{P}[\xi|\mathcal{C}]$. It is well known that a conditional
expectation is an orthogonal projection. In more details, for any $\eta
\in\mathbb{C}$,%
\[
E_{P}[(\xi\mathcal{-}\hat{\eta})\eta]=E_{P}[(\xi\mathcal{-}E_{P}%
[\xi|\mathcal{C}])\eta]=0.
\]

Does the above property still hold when $\rho$ is a sublinear operator? Note
that for any $\eta\in\mathbb{C}$,
\[%
\begin{array}
[c]{rl}
& \rho\lbrack(\xi-\hat{\eta})\eta]\\
= & \underset{P\in\mathcal{P}}{\sup}E_{P}[(\xi\mathcal{-}\hat{\eta})\eta]\\
= & \underset{P\in\mathcal{P}}{\sup}E_{P}[(\xi\mathcal{-}E_{\hat{P}}%
[\xi|\mathcal{C}])\eta]\\
\geq & 0.
\end{array}
\]
Thus, in this case, $\hat{\eta}$ is not the orthogonal projection for $\rho$.
But we notice that $\underset{\eta\in\mathbb{C}}{\inf}\rho\lbrack(\xi
-\hat{\eta})\eta]=0$. This motivates us to introduce the following definition.
For any given $\xi\in\mathbb{F}$, define $f$: $\mathbb{C}\mapsto\mathbb{R}$
by
\[
f(\tilde{\eta})=\inf_{\eta\in\mathbb{C}}\rho\lbrack(\xi-\tilde{\eta})\eta].
\]
Denote the kernel of $f$ by%
\[
\ker(f):=\{\tilde{\eta}\in\mathbb{C}\mid\;f(\tilde{\eta})=0\}.
\]

In the previous section, we prove that the minimum mean square estimator is
one element of the set $\{E_{P}[\xi|\mathcal{C}];P\in\mathcal{P}\}$. In the
following, we show that this set can be described by the kernel of $f$.

\begin{lemma}
\label{lem-kernel}If $\rho$ is a sublinear operator continuous from above on
$\mathcal{F}$, for any given $\xi\in\mathbb{F}$,
\[
\ker(f)=\{E_{P}[\xi|\mathcal{C}];P\in\mathcal{P}\}.
\]

\end{lemma}

\begin{proof}
For any $P\in\mathcal{P}$ and $\eta\in\mathbb{C}$, we have
\[
\rho\lbrack(\xi-E_{P}[\xi|\mathcal{C}])\eta]\geq E_{P}[(\xi-E_{P}%
[\xi|\mathcal{C}])\eta]=0.
\]
Then
\[
\inf_{\eta\in\mathbb{C}}\rho\lbrack(\xi-E_{P}[\xi|\mathcal{C}])\eta]\geq0.
\]
It is obvious that
\[
\inf_{\eta\in\mathbb{C}}\rho\lbrack(\xi-E_{P}[\xi|\mathcal{C}])\eta]\leq
\rho\lbrack(\xi-E_{P}[\xi|\mathcal{C}])0]=0,
\]
which leads to $\inf_{\eta\in\mathbb{C}}\rho\lbrack(\xi-E_{P}[\xi
|\mathcal{C}])\eta]=0$ for any $P\in\mathcal{P}$. Thus, $\{E_{P}%
[\xi|\mathcal{C}];P\in\mathcal{P}\}\subset\ker(f)$.

On the other hand, $\forall\tilde{\eta}\in\ker(f)$, since $\mathbb{C}$ is a
convex set and $\rho$ is a sublinear operator continuous from above, by
Mazur-Orlicz theorem, there exists a probability $\tilde{P}\in\mathcal{P}$ such
that
\[
\inf_{\eta\in\mathbb{C}}E_{\tilde{P}}[(\xi-\tilde{\eta})\eta]=\inf_{\eta
\in\mathbb{C}}\rho\lbrack(\xi-\tilde{\eta})\eta]=0.
\]
If $\tilde{\eta}\neq E_{\tilde{P}}[\xi|\mathcal{C}]$, then it is easy to find a 
$\eta^{\prime}\in\mathbb{C}$ such that $E_{\tilde{P}}[(\xi-\tilde{\eta}%
)\eta^{\prime}]<0$. Thus, $\tilde{\eta}_{0}=E_{\tilde{P}}[\xi|\mathcal{C}]$
and $\ker(f)\subset\{E_{P}[\xi|\mathcal{C}];P\in\mathcal{P}\}$.
\end{proof}

When $\mathcal{P}$ satisfies stable property which was introduced in
\cite{ADEH} (refer to \ref{append-stability} in Appendix B), we obtain
$\ker(f)$ of this case in the following theorem.

\begin{theorem}
\label{the4.1} If $\rho$ is a sublinear operator continuous from above on
$\mathcal{F}$ and the corresponding $\mathcal{P}$ is stable, then for given
$\xi\in\mathbb{F}$, $\ker(f)$ is just the set
\[
\mathbb{B}:=\{\tilde{\eta}\in\mathbb{C}\mid\mathop{\textrm{ess}\inf}_{P\in
\mathcal{P}}E_{P}[\xi|\mathcal{C}]\leq\tilde{\eta}\leq
\mathop{\textrm{ess}\sup}_{P\in\mathcal{P}}E_{P}[\xi|\mathcal{C}]\}.
\]

\end{theorem}

\begin{proof}
By Lemma \ref{lem-kernel}, $\ker(f)$ is a subset of $\mathbb{B}$. So we only
need to prove $\mathbb{B}\subset\ker(f)$.

Since $\mathcal{P}$ is `stable', for any $\eta\in\mathbb{C}$, we have
\[%
\begin{array}
[c]{r@{}l}
& \rho\lbrack(\xi-\tilde{\eta})\eta]=\rho\lbrack\underset{P\in\mathcal{P}%
}{ess\sup}E_{P}[(\xi-\tilde{\eta})\eta|\mathcal{C}]]\\
= & \rho\lbrack\eta^{+}(\underset{P\in\mathcal{P}}{ess\sup}E_{P}%
[\xi|\mathcal{C}]-\tilde{\eta})-\eta^{-}(\underset{P\in\mathcal{P}}{ess\inf
}E_{P}[\xi|\mathcal{C}]-\tilde{\eta})],
\end{array}
\]
where $\eta^{+}:=\eta\bigvee0$ and $\eta^{-}:=-(\eta\bigwedge0)$.

It yields that for any $\tilde{\eta}\in\mathbb{B}$ and $\eta\in\mathbb{C}$,
\[
\rho\lbrack(\xi-\tilde{\eta})\eta]\geq0.
\]
Then for any $\tilde{\eta}\in\mathbb{B}$,
\[
\inf_{\eta\in\mathbb{C}}\rho\lbrack(\xi-\tilde{\eta})\eta]\geq0.
\]
It is easy to see that
\[
\inf_{\eta\in\mathbb{C}}\rho\lbrack(\xi-\tilde{\eta})\eta]\leq\rho\lbrack
(\xi-\tilde{\eta})0]=0.
\]
Thus, $\inf_{\eta\in\mathbb{C}}\rho\lbrack(\xi-\tilde{\eta})\eta]=0$ for any
$\tilde{\eta}\in\mathbb{B}$ which implies that $\mathbb{B}\subset\ker(f)$.
This completes the proof.
\end{proof}

\bigskip

\subsection{A sufficient and necessary condition}

We give a sufficient and necessary condition for the minimum mean square
estimator in this subsection. Especially, we do not need to assume that $\rho$
is continuous from above on $\mathcal{F}$.

\begin{lemma}
\label{lem-optimal-characterize-1} For a given $\xi\in\mathbb{F}$, if
$\hat{\eta}$ is an optimal solution of Problem \ref{Problem}, then
\[
\rho\lbrack(\xi-\eta)(\xi-\hat{\eta})]\geq\rho(\xi-\hat{\eta})^{2}%
,\quad\forall\eta\in\mathbb{C}.
\]

\end{lemma}

\begin{proof}
For any $\eta\in\mathbb{C}$, define $f$: $[0,1]\mapsto\mathbb{R}$ by
\[
f(\lambda)=\lambda^{2}\rho(\xi-\eta)^{2}+(1-\lambda)^{2}\rho(\xi-\hat{\eta
})^{2}+2\lambda(1-\lambda)\rho\lbrack(\xi-\eta)(\xi-\hat{\eta})].
\]
It is easy to check that
\[
f(\lambda)\geq\rho\lbrack\xi-(\lambda\eta+(1-\lambda)\hat{\eta})]^{2}\geq
\rho(\xi-\hat{\eta})^{2}.
\]
This implies that $f(\lambda)$ attains the minimum on $[0,1]$ when $\lambda
=0$. Then for any $\eta\in\mathbb{C}$,
\[
f^{\prime}(\lambda)|_{0+}=-2\rho(\xi-\hat{\eta})^{2}+2\rho\lbrack(\xi
-\eta)(\xi-\hat{\eta})]\geq0,
\]
i.e.,
\[
\rho\lbrack(\xi-\eta)(\xi-\hat{\eta})]\geq\rho(\xi-\hat{\eta})^{2}%
,\quad\forall\eta\in\mathbb{C}.
\]

\end{proof}

\begin{lemma}
\label{lem-holder-ine} For any $\xi_{1},\xi_{2}\in\mathbb{F}$ such that
$\rho(|\xi_{1}|^{p})>0$ and $\rho(|\xi_{2}|^{q})>0$, we have
\[
\rho(|\xi_{1}\xi_{2}|)\leq(\rho(|\xi_{1}|^{p}))^{\frac{1}{p}}(\rho(|\xi
_{2}|^{q}))^{\frac{1}{q}},
\]
where $1<p,\,q<\infty$ with $\frac{1}{p}+\frac{1}{q}=1$.
\end{lemma}

\begin{proof}
Set
\[
X=\frac{\xi_{1}}{(\rho(|\xi_{1}|^{p}))^{\frac{1}{p}}},\quad Y=\frac{\xi_{2}%
}{(\rho(|\xi_{2}|^{q}))^{\frac{1}{q}}}.
\]
Since $|XY|\leq\frac{|X|^{p}}{p}+\frac{|Y|^{q}}{q}$, then
\[
\rho(|XY|)\leq\rho(\frac{|X|^{p}}{p}+\frac{|Y|^{q}}{q})\leq\rho(\frac{|X|^{p}%
}{p})+\rho(\frac{|Y|^{q}}{q})=1,
\]
i.e.,
\[
\rho(|\xi_{1}\xi_{2}|)\leq(\rho(|\xi_{1}|^{p}))^{\frac{1}{p}}(\rho(|\xi
_{2}|^{q}))^{\frac{1}{q}}.
\]

\end{proof}

\begin{remark}
If $\rho$ can be represented by a family of probability measures, then the
condition $\rho(|\xi_{1}|^{p})>0$ and $\rho(|\xi_{2}|^{q})>0$ can be abandoned.
\end{remark}

\begin{theorem}
\label{ns} Suppose $\underset{\eta\in\mathbb{C}}{\inf}\rho(\xi-\eta)^{2}>0$.
For a given $\xi\in\mathbb{F}$, $\hat{\eta}$ is the optimal solution of
Problem \ref{Problem} if and only if it is the bounded $\mathcal{C}%
$-measurable solution of the following equation
\begin{equation}
\inf_{\eta\in\mathbb{C}}\rho\lbrack(\xi-\hat{\eta})(\xi-\eta)]=\rho(\xi
-\hat{\eta})^{2}. \label{equ-characterize-optimal-1}%
\end{equation}

\end{theorem}

\begin{proof}
$\Rightarrow$ Since $\hat{\eta}$ is the optimal solution of Problem
\ref{Problem}, by Lemma \ref{lem-optimal-characterize-1},
\[
\inf_{\eta\in\mathbb{C}}\rho\lbrack(\xi-\hat{\eta})(\xi-\eta)]\geq\rho
(\xi-\hat{\eta})^{2}.
\]
It is obvious that
\[
\inf_{\eta\in\mathbb{C}}\rho\lbrack(\xi-\hat{\eta})(\xi-\eta)]\leq\rho
(\xi-\hat{\eta})^{2}.
\]
Then $\hat{\eta}$ is the solution of (\ref{equ-characterize-optimal-1}).

$\Leftarrow$ If $\hat{\eta}\in\mathbb{C}$ satisfying equation
(\ref{equ-characterize-optimal-1}), by Lemma \ref{lem-holder-ine}, we have
\[%
\begin{array}
[c]{l@{}l}%
\rho(\xi-\hat{\eta})^{2} & =\underset{\eta\in\mathbb{C}}{\inf}\rho\lbrack
(\xi-\hat{\eta})(\xi-\eta)]\\
& \leq\underset{\eta\in\mathbb{C}}{\inf}(\rho(\xi-\hat{\eta})^{2})^{\frac
{1}{2}}(\rho(\xi-\eta)^{2})^{\frac{1}{2}}\\
& =(\rho(\xi-\hat{\eta})^{2})^{\frac{1}{2}}[\underset{\eta\in\mathbb{C}}{\inf
}\rho(\xi-\eta)^{2}]^{\frac{1}{2}}.
\end{array}
\]
Then $\rho(\xi-\hat{\eta})^{2}\leq\underset{\eta\in\mathbb{C}}{\inf}\rho
(\xi-\eta)^{2}$. This completes the proof.
\end{proof}

\begin{remark}
\label{rem6.1} If $\rho$ is a linear operator generated by probability measure
$P$, then
\[
E_{P}[E_{P}(\xi|\mathcal{C})\eta]=E_{P}[\xi\eta],\quad\forall\eta\in
\mathbb{C}.
\]
This means $E_{P}(\xi|\mathcal{C})$ not only satisfies
(\ref{equ-characterize-optimal-1}) but also satisfies the following equation
\[
\sup_{\eta\in\mathbb{C}}E_{P}[(\xi-\hat{\eta})(\xi-\eta)]=E_{P}(\xi-\hat{\eta
})^{2}.
\]

\end{remark}

\begin{remark}
If $\rho$ can be represented by a family of probability measures, then the
condition $\underset{\eta\in\mathbb{C}}{\inf}\rho(\xi-\eta)^{2}>0$ in Theorem
\ref{ns} can be abandoned since Lemma \ref{lem-holder-ine} still holds for
either $\rho(|\xi_{1}|^{p})=0$ or $\rho(|\xi_{2}|^{q})=0$.
\end{remark}

\section{Properties of the minimum mean square estimator}

In this section, we will first give the basic properties of the minimum mean
square estimator. Then we explore the relationship between the minimum mean
square estimator and the conditional coherent risk measure and conditional $g$-expectation.

For a given $\xi\in\mathbb{F}$, we will denote the minimum mean square
estimator with respect to $\mathcal{C}$ by $\rho(\xi|\mathcal{C})$. Then
$\rho(\xi|\mathcal{C})$ satisfies the following properties.

\begin{proposition}
\label{prop-basic properties} If the sublinear operator $\rho$ is continuous
from above on $\mathcal{F}$ and proper, then for any $\xi\in\mathbb{F}$, we have:

i) If $C_{1}\leq\xi(\omega)\leq C_{2}$ for two constants $C_{1}$ and $C_{2}$,
then $C_{1}\leq\rho(\xi|\mathcal{C})\leq C_{2}$.

ii) $\rho(\lambda\xi|\mathcal{C})=\lambda\rho(\xi|\mathcal{C})$ for
$\lambda\in\mathbb{R}$.

iii) For each $\eta_{0}\in\mathbb{C}$, $\rho(\xi+\eta_{0}|\mathcal{C}%
)=\rho(\xi|\mathcal{C})+\eta_{0}$.

iv) If under each $P\in\mathcal{P}$, $\xi$ is independent of the sub $\sigma
$-algebra $\mathcal{C}$, then $\rho(\xi|\mathcal{C})$ is a constant.
\end{proposition}

\begin{proof}
i) If $C_{1}\leq\xi(\omega)\leq C_{2}$, then $C_{1}\leq E_{P}[\xi
|\mathcal{C}]\leq C_{2}$ for any $P\in\mathcal{P}$. Since $\rho(\xi
|\mathcal{C})\in\{E_{P}[\xi|\mathcal{C}];P\in\mathcal{P}\}$ (refer to the
proof of Theorem \ref{unique}), it is easy to see that $\rho(\xi|\mathcal{C})$
lies in $[C_{1},C_{2}]$.

ii) When $\lambda=0$, the statement is obvious. When $\lambda\neq0$, we have
\[
\lambda^{2}\rho(\xi-\frac{\rho(\lambda\xi|\mathcal{C})}{\lambda})^{2}%
=\rho(\lambda\xi-\rho(\lambda\xi|\mathcal{C}))^{2}=\inf_{\eta\in\mathbb{C}%
}\rho(\lambda\xi-\eta)^{2}=\lambda^{2}\inf_{\eta\in\mathbb{C}}\rho(\xi
-\eta)^{2}.
\]
It yields that
\[
\rho(\xi-\frac{\rho(\lambda\xi|\mathcal{C})}{\lambda})^{2}=\inf_{\eta
\in\mathbb{C}}\rho(\xi-\eta)^{2}.
\]
Thus, $\frac{\rho(\lambda\xi|\mathcal{C})}{\lambda}=\rho(\xi|\mathcal{C})$ due
to the uniqueness result in section 3.

iii) Note that
\[
\rho(\xi+\eta_{0}-(\eta_{0}+\rho(\xi|\mathcal{C})))^{2}=\rho(\xi-\rho
(\xi|\mathcal{C}))^{2}=\inf_{\eta\in\mathbb{C}}\rho(\xi-\eta)^{2}=\inf
_{\eta\in\mathbb{C}}\rho(\xi+\eta_{0}-\eta)^{2}.
\]
By the uniqueness of the minimum mean square estimator, we have $\rho(\xi
+\eta_{0}|\mathcal{C})=\eta_{0}+\rho(\xi|\mathcal{C})$.

iv) If under each $P\in\mathcal{P}$, $\xi$ is independent of the sub $\sigma
$-algebra $\mathcal{C}$, then $E_{P}[\xi|\mathcal{C}]$ is a constant for each
$P\in\mathcal{P}$. Since $\rho(\xi|\mathcal{C})\in\{E_{P}[\xi|\mathcal{C}%
];P\in\mathcal{P}\}$, $\rho(\xi|\mathcal{C})$ is also a constant.
\end{proof}

The coherent risk measures were introduced by Artzner et al. \cite{ADEH} and
the $g$-expectations were introduced by Peng \cite{Peng0}. The conditional coherent
risk measure and some special conditional $g$-expectations can be defined by
$\underset{P\in\mathcal{P}}{ess\sup}E_{P}[\xi|\mathcal{C}]$. In the next two
examples, we will show the minimum mean square estimator is different from the
conditional coherent risk measure and the conditional $g$-expectation.

\begin{example}
\label{exam-coherent risk}Let $\Omega=\{\omega_{1},\omega_{2}\}$,
$\mathcal{F}=\{\phi,\{\omega_{1}\},\{\omega_{2}\},\Omega\}$ and $\mathcal{C}%
=\{\phi,\Omega\}$. Set $P_{1}=\frac{1}{4}\mathbf{I}_{\{\omega_{1}\}}+\frac
{3}{4}\mathbf{I}_{\{\omega_{2}\}}$, $P_{2}=\frac{3}{4}\mathbf{I}_{\{\omega
_{1}\}}+\frac{1}{4}\mathbf{I}_{\{\omega_{2}\}}$ and $\mathcal{P}=\{\lambda
P_{1}+(1-\lambda)P_{2};\lambda\in\lbrack0,1]\}$. For each $\xi\in\mathbb{F}$,
define
\[
\rho(\xi)=\sup_{P\in\mathcal{P}}E_{P}[\xi].
\]

Set $\xi=2\mathbf{I}_{\{\omega_{1}\}}+8\mathbf{I}_{\{\omega_{2}\}}$. Then it
is easy to check
\[
\sup_{P\in\mathcal{P}}E_{P}[\xi]=6\frac{1}{2}\quad\text{and}\quad\rho
(\xi|\mathcal{C})=E_{\hat{P}}[\xi|\mathcal{C}]=5,
\]
where $\hat{P}=\frac{1}{2}\mathbf{I}_{\{\omega_{1}\}}+\frac{1}{2}%
\mathbf{I}_{\{\omega_{2}\}}$.
\end{example}

\begin{example}
\label{exam-g-expectation}Let $W(\cdot)$ be a standard $1$-dimensional
Brownian motion defined on a complete probability space $(\Omega
,\mathcal{F},P_{0})$. The information structure is given by a filtration
$\{\mathcal{F}_{t}\}_{0\leq t\leq T}$, which is generated by $W(\cdot)$ and
augmented by all the $P$-null sets. $M^{2}(0,T;\mathbb{R})$ denotes the space
of all $\mathcal{F}_{t}$-progressively measurable processes $y_{t}$ such that
$E_{P}\int_{0}^{T}|y_{t}|^{2}dt<\infty$. Let us consider the $g$-expectation
defined by the following BSDE:
\begin{equation}
y_{t}=\xi+\int_{t}^{T}|z_{s}|ds-\int_{t}^{T}z_{s}dW(s).
\label{equ-exam-g-expectation}%
\end{equation}
where $\xi$ is a bounded $\mathcal{F}_{T}$-measurable function. Here
$g(y,z)=|z|$. By the result in \cite{PP}, there exists a unique adapted pair
$(y,z)$ solves (\ref{equ-exam-g-expectation}). We call the solution $y_{t}$
the conditional $g$-expectation with respect to $\mathcal{F}_{t}$ and denote
it by $\mathcal{E}_{|z|}(\xi|\mathcal{F}_{t})$.

Consider the following linear case:
\begin{equation}
y_{t}=\xi+\int_{t}^{T}\mu_{s}z_{s}ds-\int_{t}^{T}z_{s}dW(s),
\label{equ-exam-coherent}%
\end{equation}
where $|\mu_{s}|\leq1$ $P_{0}$-a.s.. By Girsanov transform, there exists a
probability $P^{\mu}$ such that $\{y_{t}\}_{0\leq t\leq T}$ of
(\ref{equ-exam-coherent}) is a martingale under $P^{\mu}$. Set $\mathcal{P}%
:=\{P^{\mu}\mid|\mu_{s}|\leq1$ $P_{0}-$a.s.$\}$. By Theorem 2.1 in \cite{JC},
\[
\mathcal{E}_{|z|}(\xi)=\sup_{P^{\mu}\in\mathcal{P}}E_{P^{\mu}}[\xi
],\quad\forall\xi\in\mathcal{F}_{T}%
\]
and%
\[
\mathcal{E}_{|z|}(\xi|\mathcal{F}_{t})=\mathop{\mathrm{ess}\sup}_{P^{\mu}%
\in\mathcal{P}}E_{P^{\mu}}[\xi|\mathcal{F}_{t}].
\]

It is easy to see that $\mathcal{E}_{|z|}(\cdot)$ is a sublinear operator.
Denote the corresponding minimum mean square estimator by $\rho_{|z|}%
(\xi|\mathcal{F}_{t})$. We claim that the minimum mean square estimator
$\rho_{|z|}(\xi|\mathcal{F}_{t})$ does not coincide with $\mathcal{E}%
_{|z|}(\xi|\mathcal{F}_{t})$ for all bounded $\xi\in\mathcal{F}_{T}$.
Otherwise, if for all bounded $\xi\in\mathcal{F}_{T}$, $\rho_{|z|}%
(\xi|\mathcal{F}_{t})=\mathcal{E}_{|z|}(\xi|\mathcal{F}_{t})$, then by the
property ii) in Corollary \ref{prop-basic properties}, we have
\[
\mathop{\mathrm{ess}\sup}_{P^{\mu}\in\mathcal{P}}E_{P^{\mu}}[\xi
|\mathcal{F}_{t}]=\rho_{|z|}(\xi|\mathcal{F}_{t})=-\rho_{|z|}(-\xi
|\mathcal{F}_{t})=\mathop{\mathrm{ess}\inf}_{P^{\mu}\in\mathcal{P}}E_{P^{\mu}%
}[\xi|\mathcal{F}_{t}].
\]
Since the set $\mathcal{P}$ contains more than one probability measure, the
above equation can not be true for all bounded $\xi\in\mathcal{F}_{T}$. Thus,
our claim holds.
\end{example}

In the following, for simplicity, we denote $\underset{P\in\mathcal{P}%
}{ess\sup}E_{P}[\xi|\mathcal{C}]$ by $\eta_{ess}$. We first prove that
$\eta_{ess}$ is the optimal solution of a constrained mean square optimization problem.

\begin{definition}
\label{def-strict comparable} A sublinear operator $\rho$ is called strictly
comparable if for $\xi_{1}$, $\xi_{2}\in\mathbb{F}$ satisfying $\xi_{1}%
>\xi_{2}$ $P_{0}$-a.s., we have $\rho(\xi_{1})>\rho(\xi_{2})$.
\end{definition}

\begin{proposition}
\label{pro-penalize problem} Suppose that a sublinear operator $\rho$ is
continuous from above, strictly comparable and the representation set
$\mathcal{P}$ of $\rho$\ is stable. Then for a given $\xi\in\mathbb{F}$,
$\eta_{ess}$ is the unique solution of the following optimal problem:
\begin{equation}
\inf_{\eta\in\mathbb{C}}\sup_{\tilde{\eta}\in\mathbb{C}^{+}}\rho\lbrack
(\xi-\eta)^{2}+\tilde{\eta}(\xi-\eta)], \label{prob-penalized}%
\end{equation}
where $\mathbb{C}^{+}$ denotes all the nonnegative elements in $\mathbb{C}$.
\end{proposition}

\begin{proof}
We first show if the optimal solution of (\ref{prob-penalized}) exists, the
optimal solution $\hat{\eta}$ of (\ref{prob-penalized}) lies in $\mathbb{B}$,
where $\mathbb{B}:=\{\eta\in\mathbb{C};\eta\geq\eta_{ess},P_{0}-a.s.\}$.

Since
\[%
\begin{array}
[c]{r@{}l}
& \inf_{\eta\in\mathbb{C}}\sup_{\tilde{\eta}\in\mathbb{C}^{+}}\rho[(\xi
-\eta)^{2}+\tilde{\eta}(\xi-\eta)]\\
\leq & \sup_{\tilde{\eta}\in\mathbb{C}^{+}}\rho[(\xi-\eta_{ess})^{2}%
+\tilde{\eta}(\xi-\eta_{ess})]\\
\leq & \rho[(\xi-\eta_{ess})^{2}]+\sup_{\tilde{\eta}\in\mathbb{C}^{+}}%
\rho[\tilde{\eta}(\xi-\eta_{ess})]\\
= & \rho[(\xi-\eta_{ess})^{2}]<\infty,
\end{array}
\]
then the value of our problem is finite.

On the other hand, we have
\[
\sup_{\tilde{\eta}\in\mathbb{C}^{+}}\rho[\tilde{\eta}(\xi-\hat{\eta})]\leq
\sup_{\tilde{\eta}\in\mathbb{C}^{+}}\rho[(\xi-\hat{\eta})^{2}+\tilde{\eta}%
(\xi-\hat{\eta})].
\]
Since the representation set $\mathcal{P}$ is `stable', then
\[
\sup_{\tilde{\eta}\in\mathbb{C}^{+}}\rho[\tilde{\eta}(\xi-\hat{\eta}%
)]=\sup_{\tilde{\eta}\in\mathbb{C}^{+}}\rho[\tilde{\eta}(\eta_{ess}-\hat{\eta
})].
\]
If $A:=\{\omega; \hat{\eta}<\eta_{ess}\}$ is not a $P_{0}$-null set, as $\rho$
is strictly comparable, we can choose $\tilde{\eta}$ to let $\rho[\tilde{\eta
}(\xi-\hat{\eta})]$ larger than any real number. Then $P_{0}(A)=0$ and
$\hat{\eta}\geq\eta_{ess}$ $P_{0}$-a.s..

For any $\eta\in\mathbb{B}$ and $\tilde{\eta}\in\mathbb{C}^{+}$, we have
\[
\rho[(\xi-\eta)^{2}+\tilde{\eta}(\xi-\eta)]-\rho[(\xi-\eta)^{2}]\leq
\rho[\tilde{\eta}(\xi-\eta)]\leq0.
\]
Then for any $\eta\in\mathbb{B}$,
\[
\sup_{\tilde{\eta}\in\mathbb{C}^{+}}\rho[(\xi-\eta)^{2}+\tilde{\eta}(\xi
-\eta)]\leq\rho[(\xi-\eta)^{2}].
\]
On the other hand,
\[
\sup_{\tilde{\eta}\in\mathbb{C}^{+}}\rho[(\xi-\eta)^{2}+\tilde{\eta}(\xi
-\eta)]\geq\rho[(\xi-\eta)^{2}+0(\xi-\eta)]=\rho[(\xi-\eta)^{2}].
\]
Then for any $\eta\in\mathbb{B}$,
\[
\rho[(\xi-\eta)^{2}]=\sup_{\tilde{\eta}\in\mathbb{C}^{+}}\rho[(\xi-\eta
)^{2}+\tilde{\eta}(\xi-\eta)].
\]
We have
\[
\inf_{\eta\in B}\sup_{\tilde{\eta}\in\mathbb{C}^{+}}\rho[(\xi-\eta)^{2}%
+\tilde{\eta}(\xi-\eta)]=\inf_{\eta\in B}\rho[(\xi-\eta)^{2}].
\]

For any $\eta\in B$,
\[%
\begin{array}
[c]{r@{}l}
& \rho[(\xi-\eta)^{2}]=\rho[(\xi-\eta_{ess})^{2}+(\eta_{ess}-\eta)^{2}%
-2(\eta-\eta_{ess})(\xi-\eta_{ess})]\\
\geq & \rho[(\xi-\eta_{ess})^{2}+(\eta_{ess}-\eta)^{2}]-2\rho[(\eta-\eta
_{ess})(\xi-\eta_{ess})]\geq\rho[(\xi-\eta_{ess})^{2}],
\end{array}
\]
which means $\eta_{ess}$ is the best mean-square estimate among $B$.

On the other hand, for any $\eta\not \in B$, $\sup_{\tilde{\eta}\in
\mathbb{C}^{+}}\rho\lbrack(\xi-\eta)^{2}+\tilde{\eta}(\xi-\eta)]$ is equal to
$\infty$. Then $\eta_{ess}$ is also the optimal solution of
(\ref{prob-penalized}). The uniqueness is from $\rho$ is strictly comparable.
\end{proof}

Proposition \ref{pro-penalize problem} is just equivalent to say $\eta_{ess}$
is the unique solution of the following problem:
\[%
\begin{array}
[c]{r@{}c}
& \inf_{\eta\in\mathbb{C}}\rho\lbrack(\xi-\eta)^{2}]\\
& \mathrm{subject\,to}\quad\rho\lbrack(\eta_{ess}-\eta)^{+}]=0.
\end{array}
\]

Using the same method as in Proposition \ref{pro-penalize problem}, we can get
$\mathrm{ess}\inf_{P\in\mathcal{P}}E_{P\in\mathcal{P}}[\xi|\mathcal{C}]$ is
the solution of the following problem:
\[%
\begin{array}
[c]{r@{}c}
& \inf_{\eta\in\mathbb{C}}\rho\lbrack(\xi-\eta)^{2}]\\
& \mathrm{subject\,to}\quad\rho\lbrack(\eta-\mathrm{ess}\inf_{P\in\mathcal{P}%
}E_{P\in\mathcal{P}}[\xi|\mathcal{C}])^{+}]=0.
\end{array}
\]

We obtain the following necessary and sufficient condition for $\eta_{ess}$
being the minimum mean square estimator.

\begin{theorem}
\label{relation} Under the assumptions in Proposition
\ref{pro-penalize problem}, for a given $\xi\in\mathbb{F}$, $\eta_{ess}$ is
the optimal solution of Problem \ref{Problem} if and only if
\[
\inf_{\eta\in\mathbb{C}}\sup_{\tilde{\eta}\in\mathbb{C}^{+}}\rho\lbrack
(\xi-\eta)^{2}+\tilde{\eta}(\xi-\eta)]=\sup_{\tilde{\eta}\in\mathbb{C}^{+}%
}\inf_{\eta\in\mathbb{C}}\rho\lbrack(\xi-\eta)^{2}+\tilde{\eta}(\xi-\eta)].
\]

\end{theorem}

\begin{proof}
If $\eta_{ess}$ is the optimal solution of Problem \ref{Problem}, then
\[
\sup_{\tilde{\eta}\in\mathbb{C}^{+}}\inf_{\eta\in\mathbb{C}}\rho\lbrack
(\xi-\eta)^{2}+\tilde{\eta}(\xi-\eta)]\geq\inf_{\eta\in\mathbb{C}}\rho
\lbrack(\xi-\eta)^{2}]=\rho\lbrack(\xi-\eta_{ess})^{2}].
\]
On the other hand, by Proposition \ref{pro-penalize problem},
\[%
\begin{array}
[c]{r@{}l}%
\inf_{\eta\in\mathbb{C}}\sup_{\tilde{\eta}\in\mathbb{C}^{+}}\rho\lbrack
(\xi-\eta)^{2}+\tilde{\eta}(\xi-\eta)] & =\sup_{\tilde{\eta}\in\mathbb{C}^{+}%
}\rho\lbrack(\xi-\eta_{ess})^{2}+\tilde{\eta}(\xi-\eta_{ess})]\\
& \leq\rho\lbrack(\xi-\eta_{ess})^{2}]+\sup_{\tilde{\eta}\in\mathbb{C}^{+}%
}\rho\lbrack\tilde{\eta}(\xi-\eta_{ess})]\\
& =\rho\lbrack(\xi-\eta_{ess})^{2}].
\end{array}
\]
Then
\[
\inf_{\eta\in\mathbb{C}}\sup_{\tilde{\eta}\in\mathbb{C}^{+}}\rho\lbrack
(\xi-\eta)^{2}+\tilde{\eta}(\xi-\eta)]\leq\sup_{\tilde{\eta}\in\mathbb{C}^{+}%
}\inf_{\eta\in\mathbb{C}}\rho\lbrack(\xi-\eta)^{2}+\tilde{\eta}(\xi-\eta)].
\]
Since
\[
\inf_{\eta\in\mathbb{C}}\sup_{\tilde{\eta}\in\mathbb{C}^{+}}\rho\lbrack
(\xi-\eta)^{2}+\tilde{\eta}(\xi-\eta)]\geq\sup_{\tilde{\eta}\in\mathbb{C}^{+}%
}\inf_{\eta\in\mathbb{C}}\rho\lbrack(\xi-\eta)^{2}+\tilde{\eta}(\xi-\eta)]
\]
is obvious, we have
\[
\inf_{\eta\in\mathbb{C}}\sup_{\tilde{\eta}\in\mathbb{C}^{+}}\rho\lbrack
(\xi-\eta)^{2}+\tilde{\eta}(\xi-\eta)]=\sup_{\tilde{\eta}\in\mathbb{C}^{+}%
}\inf_{\eta\in\mathbb{C}}\rho\lbrack(\xi-\eta)^{2}+\tilde{\eta}(\xi-\eta)].
\]

Conversely, for any $\tilde{\eta}\in\mathbb{C}^{+}$,
\[
\inf_{\eta\in\mathbb{C}}\rho\lbrack(\xi-\eta)^{2}+\tilde{\eta}(\xi-\eta
)]=\inf_{\eta\in\mathbb{C}}\rho\lbrack(\xi+\frac{\tilde{\eta}}{2}-\eta
)^{2}-\frac{\tilde{\eta}^{2}}{4}]\leq\inf_{\eta\in\mathbb{C}}\rho\lbrack
(\xi+\frac{\tilde{\eta}}{2}-\eta)^{2}]=\inf_{\eta\in\mathbb{C}}\rho\lbrack
(\xi-\eta)^{2}].
\]
Then
\[
\sup_{\tilde{\eta}\in\mathbb{C}^{+}}\inf_{\eta\in\mathbb{C}}\rho\lbrack
(\xi-\eta)^{2}+\tilde{\eta}(\xi-\eta)]\leq\inf_{\eta\in\mathbb{C}}\rho
\lbrack(\xi-\eta)^{2}].
\]
Since
\[
\sup_{\tilde{\eta}\in\mathbb{C}^{+}}\inf_{\eta\in\mathbb{C}}\rho\lbrack
(\xi-\eta)^{2}+\tilde{\eta}(\xi-\eta)]\geq\inf_{\eta\in\mathbb{C}}\rho
\lbrack(\xi-\eta)^{2}+0(\xi-\eta)]=\inf_{\eta\in\mathbb{C}}\rho\lbrack
(\xi-\eta)^{2}]
\]
is obvious, we have
\[
\sup_{\tilde{\eta}\in\mathbb{C}^{+}}\inf_{\eta\in\mathbb{C}}\rho\lbrack
(\xi-\eta)^{2}+\tilde{\eta}(\xi-\eta)]=\inf_{\eta\in\mathbb{C}}\rho\lbrack
(\xi-\eta)^{2}].
\]
This shows $\sup_{\tilde{\eta}\in\mathbb{C}^{+}}\inf_{\eta\in\mathbb{C}}%
\rho\lbrack(\xi-\eta)^{2}+\tilde{\eta}(\xi-\eta)]$ attains its supremum when
$\tilde{\eta}=0$. By Proposition \ref{pro-penalize problem}, we also know
$\inf_{\eta\in\mathbb{C}}\sup_{\tilde{\eta}\in\mathbb{C}^{+}}\rho\lbrack
(\xi-\eta)^{2}+\tilde{\eta}(\xi-\eta)]$ attains its infimum when $\eta
=\eta_{ess}$. Since
\[
\inf_{\eta\in\mathbb{C}}\sup_{\tilde{\eta}\in\mathbb{C}^{+}}\rho\lbrack
(\xi-\eta)^{2}+\tilde{\eta}(\xi-\eta)]=\sup_{\tilde{\eta}\in\mathbb{C}^{+}%
}\inf_{\eta\in\mathbb{C}}\rho\lbrack(\xi-\eta)^{2}+\tilde{\eta}(\xi-\eta)],
\]
then
\[
\min_{\eta\in\mathbb{C}}\sup_{\tilde{\eta}\in\mathbb{C}^{+}}\rho\lbrack
(\xi-\eta)^{2}+\tilde{\eta}(\xi-\eta)]=\max_{\tilde{\eta}\in\mathbb{C}^{+}%
}\inf_{\eta\in\mathbb{C}}\rho\lbrack(\xi-\eta)^{2}+\tilde{\eta}(\xi-\eta)].
\]
By Minimax theorem, $(\eta_{ess},0)$ is the saddle point, i.e., for any
$\eta\in\mathbb{C}$ and $\tilde{\eta}\in\mathbb{C}^{+}$, we have
\[
\rho\lbrack(\xi-\eta_{ess})^{2}+\tilde{\eta}(\xi-\eta_{ess})]\leq\rho
\lbrack(\xi-\eta_{ess})^{2}]\leq\rho\lbrack(\xi-\eta)^{2}].
\]
The second inequality means $\eta_{ess}$ is the optimal solution of Problem
\ref{Problem}.
\end{proof}

\appendix

\section{Some basic results}

In this section, some results are given which are used in our paper.

\begin{theorem}
\label{theA.1} If $\rho$ is a sublinear operator and $\mathcal{P}$ is the
family of all linear operators dominated by $\rho$, then
\[
\rho(\xi)=\max_{P\in\mathcal{P}}E_{P}[\xi], \forall\xi\in\mathbb{F}.
\]

\end{theorem}

\begin{theorem}
\label{theA.2} Let $\mathbb{F}$ be a normed space and $\rho$ be a sublinear
operator from $\mathbb{F}$ to $\mathbb{R}$ dominated by some scalar multiple
of the norm of $\mathbb{F}$. Then
\[
\{x^{*}\in\mathbb{F}^{*}: x^{*}\leq\rho\quad\text{on}\quad\mathbb{F}%
\}\quad\text{is}\quad\sigma(\mathbb{F}^{*}, \mathbb{F})-\text{compact}.
\]

\end{theorem}

We denote by $\mathbb{F}_{c}^{\ast}$ the set of all countably additive measures.

\begin{theorem}
\label{theA.3} If $\mathcal{P}$ is a subset of $\mathbb{F}_{c}^{\ast}$ which
is $\sigma(\mathbb{F}^{\ast},\mathbb{F})-compact$, then there exists a
nonnegative $P_{0}\in\mathbb{F}_{c}^{\ast}$ such that the measures in
$\mathcal{P}$ are uniformly $P_{0}$-continuous. i.e., if $P_{0}(A)=0$, then
$\sup_{P\in\mathcal{P}}P(A)=0$.
\end{theorem}

\section{Some results about coherent risk measure}

In this section, we give some basic definitions and results about coherent
risk measure which are used in our paper. Reader can refer \cite{ADEH},
\cite{ADEH1} and \cite{Delbaen} for more details. Note that in order to ensure
our statements of the entire paper is consistent, the operator we used in our
paper is sublinear which is different from the coherent risk measure defined
in \cite{ADEH}, \cite{ADEH1} or \cite{Delbaen}, in which it is superlinear.
Thus it is represented as $\sup_{P\in\mathcal{P}}E_{P}$ instead of $\inf
_{P\in\mathcal{P}}E_{P}$ and the conditional expectation is taken as
$\mathrm{ess}\sup_{P\in\mathcal{P}}E_{P}[\cdot|\mathcal{C}]$ instead of
$\mathrm{ess}\inf_{P\in\mathcal{P}}E_{P}[\cdot|\mathcal{C}]$. Though the
definition is different, the methods and results are not affected.

For a given probability set $(\Omega, \mathcal{F}, P_{0})$ by the filtration
$\{\mathcal{F}_{n}\}_{n\geq1}$ such that $\mathcal{F}:=\vee_{n=1}%
\mathcal{F}_{n}$.

\begin{definition}
\label{defB.1} A map $\pi: L^{\infty}(P_{0})\mapsto\mathbb{R}$ is called a
coherent risk measure, if it satisfies the following properties:

i) Monotonicity: for all $X$ and $Y$ , if $X\geq Y$ then $\pi(X)\geq\pi(Y)$,

ii) Translation invariance: if $\lambda$ is a constant then for all $X$,
$\pi(\lambda+X)=\lambda+\pi(X)$,

iii) Positive homogeneity: if $\lambda\geq0$, then for all $X$, $\pi(\lambda
X)=\lambda\pi(X)$,

iv) Superadditivity: for all $X$ and $Y$, $\pi(X+Y)\geq\pi(X)+\pi(Y)$.
\end{definition}

\begin{definition}
\label{defB.2} The Fatou property for a risk-adjusted value $\pi$ is defined
as: for any sequence functions $(X_{n})_{n\geq1}$ such that $||X_{n}%
||_{L^{\infty}}\leq1$ and converging to $X$ in probability, then $\pi
(X)\geq\lim\sup\pi(X_{n})$.
\end{definition}

\begin{lemma}
\label{lemB.1} For any coherent risk-adjusted value $\pi$ having the Fatou
property, there exists a convex $L^{1}(P_{0})$-closed set $\mathcal{P}$ of
$P_{0}$-absolutely continuous probabilities on $(\Omega, \mathcal{F})$ called
test probabilities, such that $\pi(X)=\inf_{P\in\mathcal{P}}E_{P}[X]$.
\end{lemma}

\begin{definition}
[Stability]\label{append-stability} We say that the set $\mathcal{P}$ of test
probabilities is stable if for elements $Q^{0},Q\in\mathcal{P}^{e}$ with
associated martingales $Z_{n}^{0},Z_{n}$, where $\mathcal{P}^{e}$ denotes the
elements in $\mathcal{P}$ which is equivalent to $P_{0}$ and $Z_{n}%
^{Q}:=E_{P_{0}}[\frac{dQ}{dP_{0}}|\mathcal{F}_{n}]$. For each stopping time
$\tau$ the martingale $L$ defined as $L_{n}=Z_{n}^{0}$ for $n\leq\tau$ and
$L_{n}=Z_{\tau}^{0}\frac{Z_{n}}{Z_{\tau}}$ for $n\geq\tau$ defines an element
of $\mathcal{P}$.
\end{definition}

\begin{proposition}
\label{proB.4} A bounded $\mathcal{F}_{N}$-random variable $f$ is called a
``final value''. We consider $\Phi_{\tau}(f)=\mathrm{ess}\inf_{Q\in
\mathcal{P}^{e}}E_{Q}[f|\mathcal{F}_{\tau}]$, then the following is equivalent:

i) Stability of the set $\mathcal{P}$,

ii) Recursivity: for each final value $f$, the family $\{\Phi_{\nu}(f)|\nu\,\,
\mathrm{is\, a\, stopping\, time }\}$ satisfies: for every two stopping times
$\sigma\leq\tau$, we have $\Phi_{\sigma}(f)=\Phi_{\sigma}(\Phi_{\tau}(f))$.
\end{proposition}

\end{document}